\documentclass[12pt,letterpaper]{article}

\usepackage{amsmath}
\usepackage{amssymb}
\usepackage{amsfonts}
\usepackage{amsthm}
\usepackage{graphics}
\usepackage[pdftex]{graphicx}
\usepackage{hyperref}
\usepackage{caption}

\theoremstyle{plain}
\newtheorem{theorem}{Theorem}
\newtheorem{lemma}[theorem]{Lemma}
\newtheorem{corollary}[theorem]{Corollary}

\theoremstyle{definition}
\newtheorem{definition}{Definition}

\DeclareMathOperator{\wt}{wt}

\begin{document}

\begin{center}
\vskip 1cm{\LARGE\bf Nim Fractals}
\vskip 1cm
\large
Tanya Khovanova\\
Department of Mathematics\\
MIT\\
Cambridge, MA, 02139\\
\href{mailto:tanyakh@yahoo.com}{\tt tanyakh@yahoo.com} \\
\ \\
Joshua Xiong\\
Acton-Boxborough Regional High School\\
Acton, MA, 01719\\
\href{mailto:joshxiong7@gmail.com}{\tt joshxiong7@gmail.com}
\end{center}

\begin{abstract}
We enumerate P-positions in the game of Nim in two different ways. In one series of sequences we enumerate them by the maximum number of counters in a pile. In another series of sequences we enumerate them by the total number of counters. 

We show that the game of Nim can be viewed as a cellular automaton, where the total number of counters divided by 2 can be considered as a generation in which P-positions are born. We prove that the three-pile Nim sequence enumerated by the total number of counters is a famous toothpick sequence based on the Ulam-Warburton cellular automaton. We introduce 10 new sequences.
\end{abstract}

\section{Introduction}
\label{sec:intro}
The study of the game of Nim is fundamental to the field of combinatorial game theory. Nim is known as an \emph{impartial combinatorial game}, a game in which each player has the same moves available at each point in the game and has a complete amount of information about the game and the potential moves. In addition, there is no randomness in the game (such as rolling dice).

Originally introduced by Charles Bouton in 1901 \cite{Bouton}, Nim has played a role in many combinatorial games. The relationship between Nim and other impartial combinatorial games can be described with the Sprague-Grundy Theorem \cite{Grundy, Sprague}, which states that all impartial games are equivalent to a Nim heap. 

Although the game of Nim has been studied extensively \cite{BCG, HW}, in this paper, we invent new sequences related to enumeration of the P-positions of Nim.

In Section~\ref{sec:GameNim} we introduce the game of Nim as well as Bouton's general formula for P-positions. In Section~\ref{sec:sequences} we define the sequences we want to count and provide examples for the games with one and two piles. One set of sequences is indexed by the maximum number of counters in a P-position and the other by the total number of counters.

We continue with calculating formulae for the sequences indexed by the maximum number of counters in Section~\ref{sec:max}. We calculate the three-piles case in Section~\ref{sec:three} and the four-piles case in Section~\ref{sec:four}. It can be noted that the calculation method is different for an odd and an even number of piles. But these sections provide enough background for a general formula in Section~\ref{sec:many}.

Then we turn our attention to the sequences indexed by the total number of counters in Section~\ref{sec:total}. We start with calculated the three-piles case in Section~\ref{sec:threetotal} and discover that this sequence describes an evolution of a particular cellular automaton. We explain in Section~\ref{sec:evolution} how Nim can be viewed as an automaton. We extend the definitions to allow any impartial combinatorial game to be viewed as an automaton in Section~\ref{sec:evolany}. In Section~\ref{sec:UW} we define the Ulam-Warburton automaton, three branches of which correspond to Nim with three piles. We proceed to enumerating four piles in Section~\ref{sec:fourtotal} and arbitrarily many piles in Section~\ref{sec:manytotal}.

\section{The Game of Nim}\label{sec:GameNim}

In Nim, there are $k$ piles of counters, with $p_i$ counters in each pile. Two players alternate turns by taking some or all of the counters in a single pile. The player who takes the last counter (or equivalently, makes the last move) wins. We may denote the state or position of a game with the ordered tuple $P = (p_1, p_2, \dotsc, p_k)$. 

We begin by introducing some general definitions in game theory. Assuming that both players use an optimal strategy, there are two types of positions in a game such as Nim: 

\begin{definition}
A \emph{P-position} is a position in which the previous player will win (the one who just moved).
An \emph{N-position} is a position in which the next player will win (the one about to move).
\end{definition}

We denote the set of P-positions as $\mathcal{P}$, and the set of N-positions as $\mathcal{N}$. Thus, any move from a P-position must be an N-position, and conversely, every N-position has at least one move that results in a P-position. This motivates the following theorem \cite{ANW}:

\begin{theorem}\label{thm:general}
Suppose that the positions of a finite impartial game can be partitioned into disjoint sets $A$ and $B$ with the properties:
\begin{enumerate}
\item Every move of a position in $A$ is to a position in $B$.
\item Every position in $B$ has at least one move to a position in $A$.
\item The terminal positions are in $A$.
\end{enumerate}
\noindent Then $A = \mathcal{P}$ and $B = \mathcal{N}$.
\end{theorem}

With respect to Nim, by definition, the position $(0,0,\dotsc, 0)$ will be a (terminal) P-position. Note that the general winning strategy is to move to a P-position.

To explicitly give a formula for P-positions, we need the following definition:

\begin{definition}
The \emph{nim-sum} of two non-negative integers $x, y$ is their bit-wise XOR: $x\oplus y$. Suppose that $x = (b_j\dotsc b_2b_1)_2$ and $y =(c_j\dotsc c_2c_1)_2$ in binary with leading zeroes as necessary, where $j$ is the maximum number of digits in the binary representation of $x$ and $y$. Then the nim-sum of $x$ and $y$ is $(d_j\dotsc d_2d_1)_2$, where $d_i = b_i + c_i \pmod{2}$ for $1\le i\le j$.
\end{definition}

The nim-sum is clearly associative and commutative, and $0$ is the identity element. Further, $x\oplus y = x\oplus z$ implies $y = z$. We can extend the concept of nim-sum to a position $(p_1,p_2,\ldots,p_k)$: it is simply $p_1 \oplus p_2 \oplus \cdots \oplus p_k$.

The following theorem (see \cite{Bouton}) describes the set of P-positions in Nim.

\begin{theorem}[Bouton, 1901]\label{thm:nim}
$\mathcal{P}$ is the set of positions in the game of Nim with nim-sum $0$, and $\mathcal{N}$ is the complement.
\end{theorem}

From here, we can show that the last pile in a P-position is a function of the previous $k-1$ piles.

\begin{corollary}\label{thm:ppos}
A position $(p_1,\ldots,p_{k-1}, p_k)$ is a P-position if and only if $p_k = p_1\oplus\cdots\oplus p_{k-1}$.
\end{corollary}

\section{Nim Sequences}\label{sec:sequences}

We would like to enumerate P-positions in the game of Nim. We assume that the number of piles, $k$, is fixed, which means that piles of zero are allowed. There are two natural ways to enumerate these P-positions.

In the first set of sequences, we want to count the number of P-positions where the number of counters in each pile is bounded by some number $n$. We call these sequences \textit{indexed-by-maximum}.

\begin{itemize}
\item $a_k(n)$ is the number of P-positions in the game of Nim with $k$ piles such that each pile has no more than $n$ counters.
\item $d_k(n)$ is the number of P-positions in the game of Nim with $k$ piles such that the largest pile has exactly $n$ counters.
\end{itemize}

Note that $d_k(n)$ is the sequence of first differences of $a_k(n)$, and $a_k(n)$ is the sequence of partial sums of $d_k(n)$.

Another natural way to enumerate P-positions is to bound the total number of counters in all the piles. Note that the total number of counters in P-positions is even. We call these sequences \textit{indexed-by-total}.

\begin{itemize}
\item $A_k(n)$ is the number of P-positions in the game of Nim with $k$ piles such that the total number of counters is no more than $2n$.
\item $D_k(n)$ is the number of P-positions in the game of Nim with $k$ piles such that the total number of counters is exactly $2n$.
\end{itemize}

Once again, note that $D_k(n)$ is the sequence of first differences of $A_k(n)$, and $A_k(n)$ is the sequence of partial sums of $D_k(n)$.

In our proofs, we loosely use the term ``pile'' to refer to the number of counters in the pile.

\subsection{Relationship between sequences}

Let us denote the total number of counters in a P-position $P=(p_1,p_2,\ldots,p_k)$ as $\#(P)$; that is, $\#(P)=\textstyle\sum p_i$.

Sequences $a_k(n)$ and $A_k(n)$ can bound each other due to the following lemma:

\begin{lemma}\label{thm:numberbound}
$2\max (P) \leq \#(P) \leq k \max (P)$.
\end{lemma}

\begin{proof}
The upper bound is obvious. To prove the lower bound, consider the place values of the ones in the binary representation of $\max (P)$. By Theorem~\ref{thm:nim}, the nim-sum is $0$, and the only way for this to occur is if the other piles collectively have ones in each of those place values. Thus, the lower bound then follows immediately since the sum of the numbers other than $\max (P)$ is at least $\max (P)$.
\end{proof}

\begin{corollary}\label{thm:bounds}
$a_k(\lfloor 2n/k \rfloor) \leq A_k(n) \leq a_k(n)$.
\end{corollary}
\begin{proof}
The sequence $A_k(n)$ enumerates P-positions with no more than $2n$ counters. These positions cannot have more than $n$ counters in any pile by Lemma~\ref{thm:numberbound}, so they all are included in the enumeration corresponding to $a_k(n)$. Further, every position with the maximum pile of no more than $2n/k$ must have a sum of less than $2n$. Thus, P-positions that are counted by $a_k(\lfloor 2n/k \rfloor)$ are all included in the count of $A_k(n)$.
\end{proof}

In sections below we provide recursive formulae for sequences $a$, $d$, $A$, and $D$. For indexes of the form $2^m-1$ the formula for $a_k$ is particularly simple.

\begin{lemma}\label{thm:2m-1}
$a_k(2^m-1)=2^{m(k-1)}$.
\end{lemma}

\begin{proof}
There are $2^m$ choices for each of the first $k-1$ piles, for a total of $2^{m(k-1)}$ choices. By Corollary~\ref{thm:ppos}, the last pile is uniquely determined, and since no pile is greater than $2^m - 1$, which is the largest number that has $m$ digits in binary, the last pile will also not be greater than $2^m - 1$.
\end{proof}

Together with Corollary~\ref{thm:bounds}, this generates the following corollary,

\begin{corollary}
$$2^{(k-1)\lfloor\log_2(\lfloor 2n/k \rfloor+1)\rfloor} \leq a_k(\lfloor 2n/k \rfloor) \leq A_k(n) \leq a_k(n) \leq 2^{(k-1)\lceil\log_2(n+1)\rceil}.$$
\end{corollary}

\subsection{One or Two piles}

If there is only one pile, there is only one P-position: (0). Thus, $a_1(n) = A_1(n) = 1$ for all $n$, which is sequence A000012 in the OEIS \cite{OEIS}. Correspondingly, $d_1(0) = D_1(0)= 1$ and $d_1(n) = D_1(n) = 0$ for $n\geq 1$, which is sequence A000007.

The P-positions for the game with two piles are described by the following lemma:

\begin{lemma}\label{thm:twopiles}
A position $P$ is a P-position if and only if $P=(x,x)$ for a non-negative integer $x$.
\end{lemma}

This means that $a_2(n)= A_2(n)= n+1$, which is sequence A000027 with an initial offset of $0$. In addition, $d_2(n)=D_2(n)=1$, which is sequence A000012.

From here, the formulae for the sequences become cumbersome to express only in terms of $n$, so we define $n = 2^b - 1 + c$, where $b = \lfloor\log_2n\rfloor$, and $1\le c = n+1- 2^{\lfloor\log_2n\rfloor} \le 2^b$. In other words, $b$ is the number of digits in the binary representation of $n$, and $c-1$ is $n$ without the first digit.

\section{Indexed-by-maximum Sequences}\label{sec:max}

\subsection{Three Piles}\label{sec:three}

Consider the set of P-positions in a game of Nim with three piles. We want to find a formula for the number of such P-positions with maximum pile(s) equal $n$.

\begin{theorem}\label{thm:threed}
For $n>0$, $d_3(n) = 6c - 3$.
\end{theorem}
\begin{proof}
There are three P-positions that are permutations of $(n, n, 0)$. All other P-positions have exactly one pile of $n$. However, in order for the nim-sum to be $0$, one of the other piles must be at least $2^b$, and then the last pile is uniquely defined. There are $3$ choices for which pile has $n$ counters, $2$ choices for the pile that is at least $2^b$, and $c-1$ choices for the number of counters in this pile. This results in a total of $3 + 6(c-1) = 6c-3$ P-positions.
\end{proof}

In other words, $d_3(n) = 6(n+1- 2^{\lfloor\log_2n\rfloor})-3$. It is now sequence A241717 in the OEIS~\cite{OEIS}: 1, 3, 3, 9, 3, 9, 15, 21, 3, 9, 15, 21, 27, 33, 39, 45, 3, 9, $\ldots$.

If we arrange the numbers into a triangle as follows, the fractal-like behavior of the sequence can be seen:

\vspace{2mm}
3, 

3, 9,

3, 9, 15, 21, 

3, 9, 15, 21, 27, 33, 39, 45, 

3, 9, 15, 21, 27, 33, 39, 45, 51, 57, 63, 69, 75, 81, 87, 93,

\hspace{5cm}$\vdots$

The length of each line is a power of 2, and each line converges to A016945---the sequence $6n+3$. 

\begin{theorem}\label{thm:threea}
For $n > 0$, $a_3(n) = 2^{2b} + 3c^2$.
\end{theorem}

\begin{proof}
This statement is true for $n = 2^b - 1$ by Lemma~\ref{thm:2m-1}. Now we prove the statement for all $n = 2^b - 1 + c$. There are $2^{2b}$ P-positions such that all piles are less than $2^b$. If one of the piles is at least $2^b$, then exactly one other pile must also be at least $2^b$. The leftover pile is uniquely defined and is less than $2^b$. There are $3$ ways to designate the two piles greater than or equal to $2^b$, and $c$ different choices for each of those two piles, so the total is $2^{2b} + 3c^2$ P-positions, as desired.
\end{proof}

In other words, $a_3(n) = 2^{2\lfloor\log_2n\rfloor} + 3(n+1- 2^{\lfloor\log_2n\rfloor})^2$. It is now sequence A236305 in the OEIS~\cite{OEIS}: 1, 4, 7, 16, 19, 28, 43, 64, 67, 76, 91, 112, $\ldots$. This sequence also displays fractal-like behavior, much like the previous sequence.

\subsection{Four Piles}\label{sec:four}
When there are four piles, we cannot apply the argument used in Theorem~\ref{thm:threea} because there is a possibility that all four piles have more than $2^b - 1$ counters, and we need to make sure that each of them does not exceed $n$. However, a slight modification of our argument shows that we can find a recursive formula:

\begin{theorem}\label{thm:foura}
For $n>0$, $a_4(n) = 2^{3b} + 6c^2 2^b + a_4(c-1)$.
\end{theorem}

\begin{proof}
Suppose that all of the piles are not greater than $2^b - 1$. Similar to the argument in Theorem~\ref{thm:threea}, the first three piles that are not less than $2^b$ uniquely define a P-position where all the piles are not less than $2^b$. There are $2^{3b}$ such positions. 

In addition to that, we can have either $2$ or $4$ piles that are greater than or equal to $2^b$. If there are $2$ such piles, we can choose them in $6$ different ways, and each of those piles can be any of $c$ possible numbers. We can then choose another pile in $2^b$ ways, and the last pile will thus be fixed and less than $2^b$. This accounts for the total of $6\cdot 2^bc^2$ ways. If there are $4$ piles that are greater than $2^b-1$, we can remove $2^b$ counters from each pile without changing the nim-sum, thus reducing this situation to one when all piles are no greater than $c-1$, which can be done in $a_4(c-1)$ ways.
\end{proof}

The sequence $a_4(n)$ is now sequence A241522: 1, 8, 21, 64, 89, 168, 301, 512, 561, 712, $\ldots$.

\begin{corollary}\label{thm:fourd}
For $n>0$, $d_4(n) = (12c-6)2^b + d_4(c-1)$.
\end{corollary}

\begin{proof}
We can either use a similar argument as before or the fact that this sequence is the first difference sequence of the sequence above.
\end{proof}

The sequence $d_4(n)$ is now sequence A241718: 1, 7, 13, 43, 25, 79, 133, 211, 49, 151, 253, $\ldots$.

\subsection{Many Piles}\label{sec:many}

We will now prove a more general formula for $a_k(n)$ based on the parity of $k$.

\begin{theorem}
If $k$ is odd, $a_k(n) = \dfrac{(2^b + c)^k + (2^b - c)^k}{2^{b+1}}$, for $n>0$.
\end{theorem}

\begin{proof}
Suppose that $2i$ of the piles are at least $2^b$. There are $\dbinom{k}{2i}$ ways to choose which piles these are, and there are $c$ choices for each of these $2i$ piles. Of the remaining $k-2i$ piles, there are $2^b$ choices for the first $k-2i-1$ piles. The last pile will be uniquely  determined by Lemma~\ref{thm:ppos}, and its size will not exceed $2^b$. Hence, we get a total of $\dbinom{k}{2i} 2^{b(k-2i-1)}c^{2i}$ P-positions.

Since $2i$ can range from $0$ to $k-1$, we get the following formula:
\begin{align*}
a_k(n) &= \dbinom{k}{0} 2^{b(k-1)}c^{0} + \dbinom{k}{2} 2^{b(k-3)}c^{2} + \ldots + \dbinom{k}{k-1} 2^{0}c^{k-1}.
\end{align*}

We multiply both sides of this equation by $2^b$:

\begin{align*}
2^b a_k(n) &= \dbinom{k}{0} 2^{bk}c^{0} + \dbinom{k}{2} 2^{b(k-2)}c^{2} + \ldots + \dbinom{k}{k-1} 2^{b}c^{k-1}.
\end{align*}

Since
\begin{align*}
(2^b \pm c)^k &= \dbinom{k}{0} 2^{bk}c^{0} \pm \dbinom{k}{1} 2^{b(k-1)}c^{1} + \dbinom{k}{2} 2^{b(k-2)}c^{2} \pm \ldots,
\end{align*}
we have that
\begin{align*}
2^b a_k(n) &= \dfrac{(2^b + c)^k + (2^b - c)^k}{2}
\end{align*}
and
\begin{align*}
a_k(n) &= \dfrac{(2^b + c)^k + (2^b - c)^k}{2^{b+1}},
\end{align*}
as desired.
\end{proof}

Note that if $k=3$, we get $a_3(n)=2^{2b}+3c^2$ as expected. If $k=5$, we get $a_5(n)=2^{4b}+10\cdot2^{2b}c^2+5c^4$. This is now sequence A241523: 1, 16, 61, 256, 421, 976, 2101, 4096, 4741, $\ldots$.

We can calculate $d_k(n)$ for odd $k$ in a similar manner, or by subtracting consecutive terms:

\begin{theorem}
If $k$ is odd, 
$$d_k(n) = \dfrac{(2^b + c)^k + (2^b - c)^k-(2^b + c-1)^k-(2^b - c+1)^k}{2^{b+1}}.$$
\end{theorem}

For example, if $k=5$, $d_5(n)=10\cdot 2^{2b}(2c-1) + 20c^3-30c^2+20c-5$. This sequence is now sequence A241731: 1, 15, 45, 195, 165, 555, 1125, 1995, 645, $\ldots$.

\begin{theorem}
If $k$ is even, $a_k(n) = \dfrac{(2^b + c)^k + (2^b - c)^k - 2c^k}{2^{b+1}} + a_k(c-1)$, for $n>0$.
\end{theorem}

\begin{proof}
We can use the same argument as above for all cases except for when all of the piles are at least  $2^b$. So if we do not consider this case, there are $\dfrac{(2^b + c)^k + (2^b - c)^k - 2c^k}{2^{b+1}}$ such P-positions. Suppose, now, that all of the piles are at least $2^b$. Note that if we subtract $2^b$ from each of these piles, the nim-sum will not be changed, and now each pile is no more than $n-2^b = c-1$, so there are $a_k(c-1)$ such P-positions. So our formula is $a_k(n) = \dfrac{(2^b + c)^k + (2^b - c)^k - 2c^k}{2^{b+1}} + a_k(c-1)$.
\end{proof}

We can calculate $d_k(n)$ for even $k$ in a similar manner:

\begin{theorem}
If $k$ is even, $d_k(n)$ equals 
$$\dfrac{(2^b + c)^k + (2^b - c)^k-(2^b + c-1)^k-(2^b - c+1)^k- 2c^k+2(c-1)^k}{2^{b+1}} \allowbreak + \allowbreak d_k(c-1).$$
\end{theorem}

\section{Indexed-by-total}\label{sec:total}

We will now fix the total number of counters as $2n$. Let $\wt(n)$ denote the \textit{binary weight} of $n$; that is, the number of ones in the binary expansion of $n$.

\subsection{Three piles}\label{sec:threetotal}

\begin{theorem}\label{thm:threeD}
$D_3(n) = 3^{\wt(n)}$.
\end{theorem}

\begin{proof}
Represent each pile as a sum of distinct powers of 2. Each power of two, $2^i$, can be present in exactly two piles, or not present at all. That means if we sum all the piles we get that $n$ is the sum of powers of two that are present in exactly two piles. For each power of two that is present in the binary representation of $n$ we can chose in 3 ways in which piles they occur, for a total of $3^{\wt(n)}$ ways.
\end{proof}

This sequence is sequence A048883: 1, 3, 3, 9, 3, 9, 9, 27, 3, 9, 9, 27, 9, 27, 27, 81, 3, 9, $\ldots$. It satisfies the following recursion: $D_3(2n)=3D_3(n)$ and $D_3(2n+1)=D_3(n+1)$.

The sequence $A_3(n)$: partial sums of $D_3(n)$ is also present in the database. It is sequence A130665: 1, 4, 7, 16, 19, 28, 37, 64, 67, 76, 85, 112, 121, 148, 175, $\ldots$. The sequence satisfies the recursion: $A_3(2n)=3A_3(n-1)+ A_3(n)$ and $A_3(2n+1)=4A_3(n)$.

We calculated this sequence and discovered that it is in the database as the sequence describing the number of cells in three branches of the \textit{Ulam-Warburton cellular automaton} (see Ulam \cite{U}, Singmaster \cite{S}, Stanley and Chapman \cite{SC}, Wolfram \cite{W}). It is amazing how the On-Line Encyclopedia of Integer Sequences allows to connect different areas of mathematics.

\subsection{Evolution of Nim}\label{sec:evolution}
A natural question that arises is that if P-positions in Nim can be enumerated by cells in an automaton, can we find a bijection between P-positions of Nim and cells in the automaton? We provide the construction in this subsection.

Call a P-position $P_1$ a \textit{parent} of a P-position $P_2$, if $\#(P_1)+2=\#(P_2)$ and $P_1$ must be different from $P_2$ in exactly two piles with the same index, by one counter in each. Correspondingly, if $P_1$ is a parent of $P_2$, we call $P_2$ a \textit{child} of $P_1$. The following lemma connects the parent-child relationship to the game.

\begin{lemma}
A parent $P_1$ can be achieved in a game from P-position $P_2$.
\end{lemma}
\begin{proof}
Suppose piles $i$ and $j$ have one fewer counter in $P_1$ than in $P_2$. Then in the first move a player takes one counter from the $i$-th pile. In the next move the next player takes one counter from the $j$-th pile.
\end{proof}

The zero position: $(0,0,\ldots,0)$ does not have a parent. But any other P-position has a parent that is described by the following lemma:

\begin{lemma}
Any non-zero P-position has a parent. Each parent can be achieved by subtracting 1 from piles $i$ and $j$, where $p_i$ and $p_j$ are non-empty piles with the same number of zeros at the end of their binary representations.
\end{lemma}

\begin{proof}
The nim-sum of $p_i$ and $p_j$ should not change: $p_i \oplus p_j = (p_i-1) \oplus (p_j-1)$. This is only possible if $p_i$ and $p_j$ have the same number of zeros at the end of their binary representations since they have to regroup in the same number of places.
\end{proof}

\begin{corollary}
If there are 3 piles, each non-zero P-position has exactly one parent.
\end{corollary}

\begin{proof}
Consider the rightmost one in the binary representations of $p_1$, $p_2$ and $p_3$. This one must appear in exactly two of the representations, and so these numbers have the same number of zeroes at the end of their binary representation.
\end{proof}

Similarly, we can also describe a child.

\begin{lemma}
Each child can be achieved by adding 1 to piles $i$ and $j$, where $p_i$ and $p_j$ has the same number of ones at the end of their binary representation.
\end{lemma}

\begin{proof}
The nim-sum of $p_i$ and $p_j$ should not change: $p_i \oplus p_j = (p_i+1) \oplus (p_j+1)$. This is only possible if $p_i$ and $p_j$ have the same number of ones at the end of their binary representations since they have to regroup in the same number of places.
\end{proof}

If we play the game with 3 piles each P-position has exactly 0, 1, or 3 children.

This way we get a cellular automaton. We start with a zero position and call it alive. At each step the children of the living positions are born. Children that are born at step $n$, are called $n$-generation and $D_k(n)$ enumerates them. Similarly, $A_k(n)$ enumerates all the cells that are alive by the time $n$.

\subsection{Evolution of an impartial combinatorial game}\label{sec:evolany}

Note that we can describe an evolution of any impartial combinatorial game using the following definition. We assume that the players behave optimally. That is, if they can move to a P-position they will do so.

A P-position $P_1$ is a \textit{parent} of $P_2$ if there exists an optimal game such that $P_1$ is achieved from $P_2$ in exactly two moves in a game which takes the longest number of possible moves. 

If the longest game starting with $P_1$ takes $2n$ moves, then $n$ is the generation number of $P_1$. For Nim this definition coincides with the previous one because the longest game starting from a P-position with $2n$ counters cannot take more than $2n$ moves.

There is a standard algorithm for finding P-positions. Start with the terminal P-positions and assume they were found at step 0. Then proceed by induction. Denote the set of P-positions found at steps up to $i$ as $\mathcal{P}_i$. Denote the positions that are one move away from $\mathcal{P}_i$ as $\mathcal{N}_i$. Then the P-positions that do not belong to $\mathcal{P}_i$ and all moves from which belong to $\mathcal{N}_i$ are the P-positions from $\mathcal{P}_{i+1} \setminus \mathcal{P}_{i}$. Note that $\mathcal{P}_i \in \mathcal{P}_{i+1}$ and $\mathcal{N}_i \in \mathcal{N}_{i+1}$.

\begin{lemma}
P-positions found at step $i$ are born in generation $i$.
\end{lemma}

\begin{proof}
All optimal moves from $\mathcal{N}_i$ lead to $\mathcal{P}_i$. All moves from $\mathcal{P}_i$ lead to $\mathcal{N}_{i-1}$. Thus, if there is an optimal game where the P-position $P_1$ is reached after P-position $P_2$, then $P_1$ was found at an earlier step. That means an optimal game starting with $P_1 \in \mathcal{P}_i$ can not take more than $2i$ steps.

Now suppose $P_1 \in \mathcal{P}_{i} \setminus \mathcal{P}_{i-1}$. That means there exists a move from $P_1$ to $\mathcal{N}_{i-1} \setminus \mathcal{N}_{i-2}$. Similarly, there exists a move from $\mathcal{N}_{i-1}$ to $\mathcal{P}_{i-1} \setminus \mathcal{P}_{i-2}$. That means there exists an optimal game from $P_1$ that takes $2i$ moves, so $P_1$ is born in generation $i$.
\end{proof}

\subsection{Ulam-Warburton cellular automaton}\label{sec:UW}

Now we will describe an automaton that produces the same sequences as P-positions in the game of Nim with three piles.

Consider points on an infinite square grid on the plane. Start with the point $(0,0)$ and forbid any growth in the south branch. That is, points with coordinates $(x,y)$, where $y < 0$ and $y \leq - |x|$ are not allowed to be born. At each moment a child is born if it has exactly one alive neighbor horizontally or vertically. Remark that this corresponds to three branches of the Ulam-Warburton automaton when any direction is allowed. Figure~\ref{fig:gen6} shows 6 generations of the automaton. The dots represent cells, and the dots are connected if they form a parent-child pair. The starting cell is in the bottom center.

\begin{figure}[htbp]
  \centering
    \includegraphics[scale=0.5]{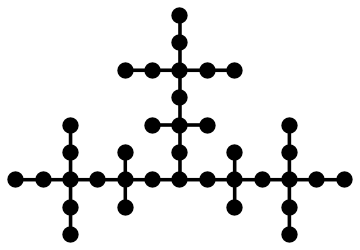}
      \caption{Ulam-Warburton automaton without the South branch after 6 generations}\label{fig:gen6}
\end{figure}

The description of points born in generation $n$ is well-known, \cite{APS, S, SC, U, W}. Suppose $n=\textstyle\sum_{j=1}^i 2^{r_j}$ for distinct integers $r_1 > r_2 > \cdots > r_i \geq 0$. Then the points that are born in generation $n$ have coordinates $\Sigma_{j=1}^i 2^{r_j} \mathbf{v}_j$, where $\mathbf{v}_j \in \{(0,1), (0,-1), (1,0), (-1,0)\}$ and $\mathbf{v}_j \neq -\mathbf{v}_{j-1}$ for $j > 1$.

We can describe three branches of this automaton in the following manner: Start in any of three directions (N, E, W) and move $2^{r_1}$ steps, then either continue forward or turn 90 degrees and move $2^{r_2}$ steps, and so on.

\begin{theorem}
The evolution graph of the game of Nim with three piles is the same as three branches of the evolution graph of the Ulam-Warburton automaton.
\end{theorem}

\begin{proof}
Consider a P-position in the game of Nim with three piles: $(p_1,p_2,p_3)$. If this P-position was born on step $n$, it means $p_1 \oplus p_2 \oplus p_3 = 0$ and $p_1+p_2+p_3 = 2n$. Suppose that we decompose each $p_i$ into distinct powers of 2. Then the powers of 2 that appear will be the $r_j$. Further, each $r_j$ is present in exactly two out of three piles. Now let us describe the ancestors of this P-position. Start with the zero position, then chose pile $i_1$ and $i_2$ in which the power $r_1$ is present. Add 1 to both piles, and continue adding $2^{r_1}$ times. Then move to the next power and so on. 

Analogously, with respect to the graph of the cellular automaton, chose a legal direction for each pair of piles. Pick a direction corresponding to the largest power of 2 in $n$ and make $2^{r_1}$ steps forward in this direction. Take the next power of 2. If it corresponds to the same two piles continue forward, otherwise turn 90 degrees either left or right depending on the new pair and move $2^{r_2}$ steps.

Now we want to make an explicit bijection between cells in the automaton and P-positions. To start, we identify the P-position $(0,1,1)$ with the point $(-1,0)$ and West direction, the P-position $(1,0,1)$ with the point $(0,1)$ and East direction, and the P-position $(1,1,0)$ with the point $(1,0)$ and North direction. Now we define turns:

\begin{itemize}
\item Left turn: changing direction from $(0,1,1)$ to $(1,1,0)$, from $(1,1,0)$ to $(1,0,1)$, and from $(1,0,1)$ to $(0,1,1)$
\item Right turn: changing direction from $(0,1,1)$ to $(1,0,1)$, from $(1,0,1)$ to $(1,1,0)$, and from $(1,1,0)$ to $(0,1,1)$.
\end{itemize}

Each cell in the automaton (correspondingly, P-position) has exactly one parent. The cell (P-position) is uniquely described by the path from the starting point (terminal position). We showed the bijection between the paths which establishes the bijection between the cells and the P-positions.

\end{proof}

For example, consider the P-position $(14,11,5)$, which can be decomposed into powers of 2: $(8+4+2, 8+2+1,4+1)$. This means that the evolution happens in the following way. Start with the P-position $(0,0,0)$, then 8 generations are born in the direction $(1,1,0)$ until the P-position $(8,8,0)$ is reached. After that 4 generations are born in the direction $(1,0,1)$ until the P-position $(12,8,4)$ is reached. After that 2 generations are born in the direction $(1,1,0)$ reaching $(14,10,4)$, then the child is born in the direction $(0,1,1)$ reaching the final destination $(14,11,5)$. This corresponds to the following walk on the automaton: 8 steps to the right until the coordinates $(8,0)$, turn right, 4 more steps reaching $(8,-4)$, turn left and make 2 more steps reaching $(10,-4)$, then turn left again and make one step to get to $(10,-3)$.

It is more natural to place the Nim evolution in 3D, but such a graph is more difficult to draw and to visualize; see Figure~\ref{fig:NimEvol}.

\begin{figure}[htbp]
  \centering
    \includegraphics[scale=0.5]{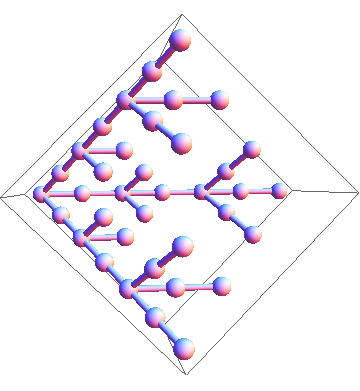}
      \caption{Nim evolution after 6 generations in 3D}\label{fig:NimEvol}
\end{figure}

\subsection{Four piles}\label{sec:fourtotal}

Let us move to four piles. The number of possible parents is 1, 2 or 6:

\begin{itemize}
\item 1: if there is one pair of binary numbers with the same number of zeros at the end,
\item 2: if there are two pairs of binary numbers such that the number of zeros at the end is the same within each pair and different for different pairs,
\item 6: all four binary numbers have the same number of zeros at the end.
\end{itemize}

Similarly, the number of possible children is 1, 2, or 6. Table~\ref{tab:cp} shows examples of P-positions with different numbers of parents and children.

\begin{table}[htbp]
\centering
\begin{tabular}{| c | c | c | c |}
  \hline
   & 1 child & 2 children & 6 children\\
  \hline                       
  1 parent & (0,1,2,3) & (0,0,1,1) & (0,0,2,2)\\
  2 parents & (0,1,4,5) & (1,1,2,2) & (2,2,4,4)\\
  6 parents & (1,3,5,7) & (1,1,3,3) & (1,1,1,1)\\ 
  \hline  
\end{tabular}
\caption{P-positions with different numbers of parents/children}\label{tab:cp}
\end{table}

Suppose the total number of counters is $2n$. We computed the sequence $D_4(n)$, which is now sequence A237711: 1, 6, 7, 36, 13, 42, 43, 216, 49, 78, 55, 252, 85, $\ldots$. Sequence $D_4(n)$ can be described recursively:

\begin{lemma}
$D_4(0)=1$, $D_4(1)=6$, $D_4(2n+1)=6D_4(n)$ for $n\geq 0$, and $D_4(2n+2)=D_4(n+1)+D_4(n)$ for $n \geq 0$
\end{lemma}

\begin{proof}
Consider a P-position with total sum $4n+2$. It has exactly two odd piles. That means it can be achieved by doubling a P-position with total sum $2n$ and adding two counters in any pair of piles. The latter can be done in 6 ways. Therefore, $D_4(2n+1)=6D_4(n)$.

Now consider a P-position with total sum $4n+4$. It can have either four odd piles or four even piles. If it has four odd piles, this P-position can be achieved by doubling all the piles in a P-position with total sum $2n$ and adding one counter to each pile. If it has four even piles, this P-position can be achieved by doubling a P-position with total sum $2n+2$. Therefore, $D_4(2n+2)=D_4(n+1)+D_4(n)$.
\end{proof}

The corresponding sequence of partial sums $A_4(n)$ is now sequence A237686: 1, 7, 14, 50, 63, 105, 148, 364, 413, 491, 546, 798, 883, 1141, ... This sequence can also be described recursively:

\begin{lemma}
$A_4(0)=1$, $A_4(1)=7$, $A_4(2n+1)=7A_4(n) + A_4(n-1)$ for $n\geq 1$, and $A_4(2n+2)=7A_4(n) + A_4(n+1)$ for $n \geq 1$.
\end{lemma}

\begin{proof}
We use induction to prove this statement. We can see that the base cases hold via direct computation. Now assume that the recurrence relation holds for $k\leq 2n$. By definition, $A_4(2n+1)=A_4(2n) + D_4(2n+1)$. Using the inductive hypothesis, $A_4(2n+1)=7A_4(n-1) + A_4(n) + 6D_4(n) = A_4(n-1) + 6(A_4(n-1)+D_4(n)) + A_4(n)= 7A_4(n) + A_4(n-1)$.

Similarly, by definition: $A_4(2n+2)=A_4(2n+1) + D_4(2n+2)$. Using the previous result, $A_4(2n+2)=7A_4(n) + A_4(n-1) + D_4(n)+ D_4(n+1)=7A_4(n) + A_4(n+1)$, which completes the induction.
\end{proof}

\subsection{Many piles}\label{sec:manytotal}

We now calculate these sequences for any number of piles. First, we compute the initial terms of $D_k$.

\begin{lemma}
$D_k(0)=1$, $D_k(1) = \dbinom{k}{2}$.
\end{lemma}

\begin{proof}
It is easy to see that $D_k(0) = 1$ because this is just the position $(0,\ldots,0)$. In the case of $D_k(1)$, we can choose two piles to have one counter each in $\dbinom{k}{2}$ ways, and this is the only way for this position to have a nim-sum of zero.
\end{proof}

The following theorem provides a recursive formula for $D_k(n)$.
\begin{theorem} Assuming $D_k(j)=0$ for negative $j$:

$D_k(2n+1)=\dbinom{k}{2} D_k(n) + \dbinom{k}{6} D_k(n-1) + \dbinom{k}{10} D_k(n-2) + \ldots$,

$D_k(2n+2)=\dbinom{k}{0} D_k(n+1) + \dbinom{k}{4} D_k(n) + \dbinom{k}{8} D_k(n-1) + \ldots$,
\end{theorem}

\begin{proof}
We exploit the fact that each P-position with only even piles and total sum $2m$ can be realized by doubling each pile in a corresponding P-position with total sum $m$, which means that there is a bijection. 

If the total number of counters is $4n+2$, then the number of odd piles could be 2, 6, 10 and so on. If there are $i$ odd piles, then we can choose which piles are odd in $\dbinom{k}{i}$ ways. Then we can remove one counter from every odd pile and divide each pile by 2. Using the bijection above, there will be $\dbinom{k}{i}D_k(n+1-\frac{i+2}{4})$ such P-positions for each choice of $i$. The case $4n+4$ is similar.
\end{proof}

For example, if $k=5$, then, $D_5(0)=1$, $D_5(1)=10$, $D_5(2n+1)= 10D_5(n)$, and $D_5(2n+2)=D_5(n+1)+5D_5(n)$. This is now sequence A238759: 1, 10, 15, 100, 65, 150, 175, 1000, 565, $\ldots$. 

Similarly, we can prove a recursive formula for $A_k(n)$.

\begin{theorem}
$A_k(2n+1)=\left(\dbinom{k}{2}+\dbinom{k}{0}\right) A_k(n) + \left(\dbinom{k}{6}+\dbinom{k}{4}\right) A_k(n-1) + \left(\dbinom{k}{10}+\dbinom{k}{8}\right) A_k(n-2) + \ldots$,

$A_k(2n+2)=\dbinom{k}{0} A_k(n+1) + \left(\dbinom{k}{2}+\dbinom{k}{4}\right) A_k(n) + \allowbreak  \left(\dbinom{k}{6}+\dbinom{k}{8}\right) \allowbreak A_k(n-1) + \ldots$,
\end{theorem}

\begin{proof}
We can show this by using the partial sums of $D_k$.
\end{proof}

For example, if $k=5$, then, $A_5(0)=1$, $A_5(1)=11$, $A_5(2n+1)= 11A_5(n)+5A_5(n-1)$, and $A_5(2n+2)=A_5(n+1)+15A_5(n)$. This is now sequence A238147: 1, 11, 26, 126, 191, 341, 516, 1516, 2081, $\ldots$.

\section{Acknowledgements}

We are grateful to the MIT-PRIMES program for supporting this research.

\bigskip
\hrule
\bigskip

\noindent 2010 {\it Mathematics Subject Classification}: Primary 91A46; 68Q80.

\noindent \emph{Keywords: } combinatorial games, Nim, P-positions, cellular automata.

\bigskip
\hrule
\bigskip

\noindent
(Mentions A000007, A000012, A000027, A016945, A048883, A130665. New sequences A236305, A237686, A237711, A238147, A238759, A241522, A241523, A241717, A241718, A241731)

\bigskip
\hrule
\bigskip

\end{document}